\documentclass[12pt,reqno]{amsart}

\usepackage{amssymb,amsmath,graphicx,amsfonts,euscript}
\usepackage{color}

\newtheorem{thm}{Theorem}[section]

\newtheorem{prop}[thm]{Proposition}
\newtheorem{define}[thm]{Definition}

\newtheorem{lemma}[thm]{Lemma}

\newcommand{\p}{\partial}
\newcommand{\om}{\omega}
\newcommand{\na}{\nabla}

\numberwithin{equation}{section}

\begin{document}
\title[Damped 2D Boussinesq equations]
{Small global solutions to the damped two-dimensional Boussinesq equations}

\author[D. Adhikari, C. Cao, J. Wu, X. Xu]{Dhanapati Adhikari$^{1}$, Chongsheng Cao$^{2}$,
Jiahong Wu$^{3}$, Xiaojing Xu$^{4}$}

\address{$^1$ Mathematical Sciences Department, Marywood University,
Scranton, PA 18509, USA.}

\email{dadhikari@marywood.edu}

\address{$^2$ Department of Mathematics,
Florida International University,
Miami, FL 33199, USA.}

\email{caoc@fiu.edu}

\address{$^3$ Department of Mathematics,
Oklahoma State University,
401 Mathematical Sciences,
Stillwater, OK 74078, USA; and
Department of Mathematics,
College of Natural Science,
Chung-Ang University,
Seoul 156-756, Korea}

\email{jiahong@math.okstate.edu}

\address{$^4$  School of Mathematical Sciences, Beijing Normal
University and Laboratory of Mathematics and Complex Systems, Ministry of
Education, Beijing 100875, P. R. China.}

\email{xjxu@bnu.edu.cn}

\vskip .2in
\begin{abstract}
The two-dimensional (2D) incompressible Euler equations have been thoroughly investigated and
the resolution of the global (in time) existence and uniqueness issue is currently in a satisfactory status.
In contrast, the global regularity problem concerning the 2D inviscid Boussinesq equations remains
widely open. In an attempt to understand this problem, we examine the damped 2D Boussinesq equations and
study how damping affects the regularity of solutions. Since the damping effect is
insufficient in overcoming the difficulty due to  the ``vortex stretching", we seek unique global small
solutions and the efforts have been mainly devoted to minimizing the smallness assumption. By positioning
the solutions in a suitable functional setting (more precisely the homogeneous Besov space
$\mathring{B}^1_{\infty,1}$), we are able to obtain a unique global solution under a minimal smallness
assumption.
\end{abstract}

\maketitle

\section{Introduction}

This paper examines the global (in time) existence and uniqueness problem on the incompressible
2D Boussinesq equations with damping
\begin{equation} \label{dbouss}
\begin{cases}
\p_t u + (u\cdot\nabla) u + \nu u = -\nabla p + \theta \mathbf{e}_2, \quad x\in \mathbb{R}^2,\, t>0,\\
\p_t \theta + (u\cdot\nabla) \theta + \lambda \theta =0, \quad x\in \mathbb{R}^2,\, t>0,\\
\nabla \cdot u =0, \quad x\in \mathbb{R}^2,\, t>0,\\
u(x,0) =u_0(x), \quad \theta(x,0) =\theta_0(x),  \quad x\in \mathbb{R}^2,
\end{cases}
\end{equation}
where $u$ represents the fluid velocity, $p$ the pressure, $\mathbf{e}_2$ the unit
vector in the vertical direction, $\theta$ the temperature in thermal convection
or the density in geophysical flows, and $\nu>0$ and $\lambda>0$ are real parameters.
When $\nu u$ is replaced by $-\nu \Delta u$ and $\lambda \theta$ by
$-\lambda\Delta \theta$,\, (\ref{dbouss}) becomes the standard viscous Boussinesq
equations. (\ref{dbouss}) with $\nu=0$ and $\lambda=0$ reduces to the inviscid
2D Boussinesq equations. If $\theta$ is identically zero, (\ref{dbouss}) degenerates to the
2D incompressible Euler equations.

\vskip .1in
The Boussinseq equations model many geophysical flows such as atmospheric fronts and ocean
circulations (see, e.g., \cite{Con_D,Gill,Maj,Pe}). Mathematically the 2D Boussinesq equations
serve as a lower-dimensional model of the 3D hydrodynamics equations. In fact, the
2D Boussinesq equations retain some key features of the 3D Euler and Navier-Stokes
equations such as the vortex stretching mechanism. The vortex stretching term
is the greatest obstacle in dealing with the global regularity issue
concerning the Boussinesq equations. When suitable partial dissipation or fractional
Laplacian dissipation with sufficiently large index is added, the vortex stretching can be
controlled and the global regularity can be established (see, e.g., \cite{ACW10,ACW11,BSch,CaoWu1,Ch,ChaeWu,CV,
DP3,HS,Hmidi,HKR1,HKR2,HL,KRTW,Lai,LLT,MX,Mof,WangZ,Xu,Zhao}). In contrast, the global regularity problem
on the inviscid Boussinesq equations appears to be out of reach in spite of the progress on the local
well-posedness and regularity criteria (see, e.g., \cite{ChN,Cui,Dan,ES,MB,Mof,Oh,XY}).
This work is partially aimed at understanding
this difficult problem by examining how damping affects the regularity of the solutions to the
Boussinesq equations.

\vskip .1in
As we know, the issue of global existence and uniqueness relies crucially on whether or not one
can obtain global bounds on the solutions. Thanks to the
divergence-free condition $\nabla \cdot u=0$, global {\it a priori} bounds for $\theta$ in any Lebesgue
space $L^q$ and $u$ in $L^2$ follow directly from simple energy estimates,
$$
\|\theta(t)\|_{L^q} \le \|\theta_0\|_{L^q}, \qquad
\|u(t)\|_{L^2} \le \|u_0\|_{L^2} + t\,\|\theta_0\|_{L^2}
$$
for $1\le q\le \infty$. However, global bounds for $(u, \theta)$ in any Sobolev space, say $H^1$, can not be
easily achieved and the difficulty comes from the vortex stretching term. More precisely, if we resort to
the equations of the
vorticity $\om$ and $\nabla^\perp \theta$
\begin{equation} \label{vorticity}
\begin{cases}
\p_t \om + (u\cdot\nabla) \om + \nu \om = \partial_{x_1} \theta, \\
\p_t (\na^\perp \theta) + (u\cdot\nabla) (\na^\perp\theta)
+ \lambda \na^\perp\theta =(\na^\perp \theta\cdot \na) u,
\end{cases}
\end{equation}
we unavoidably have to deal with the ``vortex stretching term" $(\na^\perp \theta\cdot \na) u$, which appears to elude
any suitable bound. Here $\na^\perp =(-\partial_{x_2}, \partial_{x_1})$.
The damping terms are not sufficient to overcome this difficulty. Therefore damping does not appear to make a big
difference in dealing with solutions emanating from a general data.

\vskip .1in
The aim here is at the global existence and uniqueness of small solutions. We remark that,
if we consider solutions of (\ref{dbouss}) with initial data in the classical setting, say $(u_0, \theta_0)\in H^s$
with $s>2$,
then it is not difficult to prove the global existence of solutions when we impose a very strong smallness
condition such as
\begin{equation}\label{small1}
1+ \|u_0\|_{H^s} + \|\theta_0\|_{H^s} \le C\min\{\nu, \lambda\},
\end{equation}
where $C$ is a suitable constant independent of $\nu$ and $\lambda$. In fact, the global regularity follows easily
from the local well-posedness in $H^s$ and the global energy inequality for
$Y(t)\equiv \|u(t)\|_{H^s} + \|\theta(t)\|_{H^s}$,
\begin{equation}\label{energy0}
\frac{d}{dt} Y(t) + \min\{\nu, \lambda\}\, Y(t) \le C (1 + \|u\|_{H^s} + \|\theta\|_{H^s})\,Y(t).
\end{equation}
However, the smallness assumption (\ref{small1}) appears to be too restrict.
In particular, it forces $\nu$ and $\lambda$ to be of order 1. It appears that (\ref{small1})
can not be easily weakened if we seek solutions in the classical functional setting.
This is due to the presence of the forcing term $\theta \mathbf{e}_2$ in the velocity equation and the
growth of $\|u(t)\|_{H^s}$ in time. In fact, as shown by Brandolese
and Schonbek for the 3D viscous Boussinesq equations, the $L^2$-norm of $u$ may grow in time when
the spatial integral of $\theta_0$ is not zero and when $\theta$ does not decay sufficiently
fast in time \cite{BSch}. Our efforts have been devoted to seeking a suitable functional
setting so that (\ref{small1}) can be relaxed. The right functional space is the homogeneous
Besov space $\mathring{B}^1_{\infty,1}$ and our main result can be stated as follows.

\begin{thm} \label{main}
Consider (\ref{dbouss}) with $\nu>0$ and $\lambda>0$. Assume that
$(u_0,\theta_0) \in L^2$ obeys the smallness conditions
\begin{equation}\label{smallness}
\|\nabla u_0\|_{\mathring{B}^0_{\infty,1}} < A_0 \equiv \min\left\{\frac\nu{2C_0}, \frac\lambda{C_0}\right\}
\quad \mbox{and}\quad \|\nabla \theta_0\|_{\mathring{B}^0_{\infty,1}}
< B_0\equiv \frac{\nu}{2C_0} \|\nabla u_0\|_{\mathring{B}^0_{\infty,1}}
\end{equation}
for a suitable constant $C_0$ independent of $\nu$ and $\lambda$. Then (\ref{dbouss}) has a unique global solution $(u, \theta)$
satisfying
\begin{equation}\label{regclass}
 (u, \theta)\in L^\infty([0,\infty); L^2), \quad \nabla u, \nabla \theta \in L^\infty([0,\infty); \mathring{B}^0_{\infty,1}).
\end{equation}
In addition,
$$
\sup_{t\geq 0}\|\nabla u(t)\|_{\mathring{B}^0_{\infty,1}} <  A_0 \quad \mbox{and}\quad
\quad \sup_{t\geq0}\|\nabla \theta(t)\|_{\mathring{B}^0_{\infty,1}} < B_0.
$$
\end{thm}

More details on the homogeneous Besov space can be found in Appendix A. We remark that the smallness
condition (\ref{smallness}) is weaker than (\ref{small1}) in two senses: first,
the norm in $\mathring{B}^1_{\infty,1}$ (the smallness for $(\nabla u,\nabla \theta)$ in $\mathring{B}^0_{\infty,1}$
is equivalent to the smallness of $(u,\theta)$ in $\mathring{B}^1_{\infty,1}$) is weaker than the norm in $H^s$
due to the embedding $H^s(\mathbb{R}^2)\hookrightarrow \mathring{B}^1_{\infty,1}(\mathbb{R}^2)$ for $s>2$; and
second, (\ref{smallness}) does not include the factor $1$,
as in (\ref{small1}). $\mathring{B}^1_{\infty,1}$ appears to be a very natural setting if one wants to ensure the uniqueness
of the solutions. It may be difficult to further weaken the functional setting.

\vskip .1in
May it be possible to sharpen the result of Theorem \ref{main} by removing one of the damping terms
$\nu\, u$ or $\lambda\, \theta$?  This problem appears to be extremely challenging. For
the 2D Boussinesq equations, it appears that any functional setting guaranteeing the uniqueness of
solutions necessarily involves
the derivatives of the functions.  If $\lambda\, \theta$ is not present, the norm of $\theta$ in such a functional
setting may grow exponentially (in time) at the rate of $\|\nabla u\|_{L^\infty}$. Consequently the norm
of $u$ may grow even when the velocity equation has the damping term $\nu u$. It is also clear that, if
$\nu\, u$ is missing, then the norm of $u$ is expected to grow. Therefore, when any one of the damping terms
is removed, the small data well-posedness problem becomes as difficult as the well-posedness problem
for a general initial data.

\vskip .1in
The rest of this paper consists of a section that proves Theorem \ref{main} and an appendix that provides
the definitions of Besov spaces and related facts.

\vskip .3in
\section{Proof of Theorem \ref{main}}

This section is devoted to the proof of Theorem \ref{main}. The proof is lengthy and consists
of five major steps. The first step  constructs a sequence of approximate smooth solutions
while the second step shows that these approximate solutions obey global (in time) bounds
in the functional setting of the initial data.  One key component leading to the global bounds
is a global differential inequality, which we establish as a proposition. The third step is to
show that the sequence of approximate solutions consists of a strongly convergent subsequence.
The fourth step shows that the limit of the convergence actually solves the Boussinesq equations
in a suitable functional setting. This step involves extensive applications of the Besov spaces
techniques. The last step asserts the uniqueness of the solutions.

\vskip .1in
As a preparation for the proof of Theorem \ref{main}, we first state and prove a global differential
inequality.

\begin{prop}\label{globalbd}
Consider (\ref{dbouss}) with $\nu>0$ and $\lambda>0$. Assume $(u_0, \theta_0) \in L^2$
and $(\nabla u_0, \nabla \theta_0)\in \mathring{B}^0_{\infty,1}$.
Assume that  $(u, \theta)$ is the corresponding solution.
Then $(u, \theta)$ satisfies the differential
inequality, for any $q\in [2, \infty]$,
\begin{eqnarray}
&& \frac{d}{dt} \|\nabla u\|_{\mathring{B}^0_{q,1}} + \nu \|\nabla u\|_{\mathring{B}^0_{q,1}}
\le C_0\,\|\nabla u\|_{\mathring{B}^0_{\infty,1}}\,\|\nabla u\|_{\mathring{B}^0_{q,1}}
+ C_0\,\|\nabla \theta\|_{\mathring{B}^0_{q,1}}, \nonumber\\
&& \frac{d}{dt} \|\nabla\theta\|_{\mathring{B}^0_{q,1}} + \lambda \|\nabla\theta\|_{\mathring{B}^0_{q,1}}
\le C_0\,\|\nabla u\|_{\mathring{B}^0_{\infty,1}}\,\|\nabla\theta\|_{\mathring{B}^0_{q,1}},\label{apin}
\end{eqnarray}
where $C_0>0$ is a constant independent of $q$, $\nu$ and $\lambda$.
\end{prop}

\begin{proof}[Proof of Proposition \ref{globalbd}]
Let $2\le q\le \infty$ be arbitrarily fixed. We now derive the differential
inequalities for $\|\na \theta\|_{\mathring{B}^0_{q,1}}$
and $\|\na u\|_{\mathring{B}^0_{q,1}}$. Let $j$ be an integer. Applying $\Delta_j$ to the equation of
$\nabla \theta$ yields
$$
\partial_t\Delta_j \nabla\theta + \Delta_j\nabla((u\cdot \nabla)\theta)  +  \lambda\Delta_j\nabla \theta =0.
$$
Multiplying by $\Delta_j\nabla\theta|\Delta_j\nabla\theta|^{q-2}$ and integrating in space, we have
\begin{eqnarray}
\frac1q \frac{d}{dt} \|\Delta_j\nabla\theta\|_{L^q}^q + \lambda \|\Delta_j\nabla\theta\|_{L^q}^q =K_1 + K_2, \label{k1k2}
\end{eqnarray}
where
\begin{eqnarray*}
K_1 &=& -\int \Delta_j\nabla\theta|\Delta_j\nabla\theta|^{q-2}\cdot\, \Delta_j(\nabla u\cdot \nabla\theta), \\
K_2 &=& -\int \Delta_j\nabla\theta|\Delta_j\nabla\theta|^{q-2}\cdot\, \Delta_j(u\cdot\nabla\nabla \theta).
\end{eqnarray*}
To bound $K_1$, we use the notion of paraproducts to write $K_1$ into three terms,
$$
K_1 = K_{11} + K_{12} + K_{13},
$$
where
\begin{eqnarray*}
K_{11} &=& -\sum_{|k-j|\le 2}\int \Delta_j\nabla\theta|\Delta_j\nabla\theta|^{q-2}\cdot\,\Delta_j(S_{k-1}\nabla u\cdot \Delta_k \na\theta),\\
K_{12} &=& -\sum_{|k-j|\le 2}\int \Delta_j\nabla\theta|\Delta_j\nabla\theta|^{q-2}\cdot\,\Delta_j(\Delta_k\nabla u\cdot S_{k-1}\na \theta),\\
K_{13} &=& -\sum_{k\ge j-1} \int \Delta_j\nabla\theta|\Delta_j\nabla\theta|^{q-2}\cdot\,\Delta_j(\Delta_k\nabla u\cdot\widetilde{\Delta}_k \na\theta)
\end{eqnarray*}
with $\widetilde{\Delta}_k =\Delta_{k-1} + \Delta_k + \Delta_{k+1}$. By H\"{o}lder's inequality,
$$
|K_{11}| \le C\,\|\Delta_j\nabla\theta\|_{L^q}^{q-1}\, \|S_{j-1}\nabla u\|_{L^\infty}\, \|\Delta_j\nabla\theta\|_{L^q}
\le C\, \|\nabla u\|_{L^\infty}\, \|\Delta_j\nabla\theta\|_{L^q}^q,
$$
where $C$ is a constant independent of $q$. Here we have applied the simple fact that,
for fixed $j$, the summation in $K_{11}$ is for a finite number of $k$'s satisfying
$|k-j|\le 2$ and the estimate for the term with the index $k$ is only a constant
multiple of the bound for the term
with the index $j$. By H\"{o}lder's inequality,
$$
|K_{12}| \le C\,\|\Delta_j\nabla\theta\|_{L^q}^{q-1}\,\|\Delta_j \nabla u\|_{L^\infty}\,
\sum_{m\leq j-1} \|\Delta_m \nabla\theta\|_{L^q},
$$
where again $C$ is independent of $q$. Thanks to $\na \cdot u =0$ and by H\"{o}lder's inequality
and Bernstein's inequality,
\begin{eqnarray*}
|K_{13}| &\le& C\,\|\Delta_j\nabla\theta\|_{L^q}^{q-1}\,\sum_{k\ge j-1} 2^j \|\Delta_k \nabla u\|_{L^\infty}\,\|\Delta_k\theta\|_{L^q} \\
&\le& C\,\|\Delta_j\nabla\theta\|_{L^q}^{q-1}\,\sum_{k\ge j-1} 2^{j-k}\,\|\Delta_k \nabla u\|_{L^\infty}\,\|\Delta_k\nabla \theta\|_{L^q}.
\end{eqnarray*}
We now turn to $K_2$. We decompose it into five terms via the notion of paraproducts,
$$
K_2 =K_{21} + K_{22} + K_{23} + K_{24} + K_{25},
$$
where
\begin{eqnarray*}
K_{21} &=& -\sum_{|k-j|\le 2}\int \Delta_j\nabla\theta|\Delta_j\nabla\theta|^{q-2}\cdot\,[\Delta_j, S_{k-1}u\cdot\nabla]\Delta_k\na \theta,\\
K_{22} &=& -\sum_{|k-j|\le 2}\int \Delta_j\nabla\theta|\Delta_j\nabla\theta|^{q-2}\cdot\, (S_{k-1}u -S_j u)\cdot\nabla\Delta_j\Delta_k \na\theta,\\
K_{23} &=& -\int \Delta_j\nabla\theta|\Delta_j\nabla\theta|^{q-2}\cdot\,S_j u\cdot\nabla \Delta_j \na\theta,\\
K_{24} &=& -\sum_{|k-j|\le 2}\int \Delta_j\nabla\theta|\Delta_j\nabla\theta|^{q-2}\cdot\,\Delta_j (\Delta_k u\cdot\nabla S_{k-1}\nabla \theta),\\
K_{25} &=& -\sum_{k\ge j-1} \int \Delta_j\nabla\theta|\Delta_j\nabla\theta|^{q-2}\cdot\,\Delta_j(\Delta_k u\cdot\na \widetilde{\Delta}_k \na\theta).
\end{eqnarray*}
Thanks to the divergence-free condition, $\na \cdot u=0$, we have $K_{23}=0$. By H\"{o}lder's inequality
and a standard commutator estimate (see, e.g., \cite[p.110]{BCD}),
$$
|K_{21}| \le C\, \|\Delta_j\nabla\theta\|_{L^q}^{q-1}\, \|\na S_{j-1} u\|_{L^\infty}\, \|\Delta_j\na\theta\|_{L^q}\le C\, \|\nabla u\|_{L^\infty}\, \|\Delta_j\nabla\theta\|_{L^q}^q,
$$
where $C$ is a constant independent of $q$. It is easy to see from H\"{o}ler's inequality and Bernstein's inequality that
$$
|K_{22}| \le C\, \|\Delta_j\nabla\theta\|_{L^q}^{q-1}\, \|\Delta_j\na u\|_{L^\infty} \, \|\Delta_j \na\theta\|_{L^q}.
$$
Again, by H\"{o}ler's inequality and Bernstein's inequality,
$$
|K_{24}| \le C\, \|\Delta_j\nabla\theta\|_{L^q}^{q-1}\,\|\Delta_j\na u\|_{L^\infty} \,
\sum_{m\le j-1}\|\Delta_m\na \theta\|_{L^q}.
$$
Due to the divergence-free condition,
$$
|K_{25}| \le C\, \|\Delta_j\nabla\theta\|_{L^q}^{q-1}\, \sum_{k\ge j-1} 2^{j-k} \|\na \Delta_k u\|_{L^\infty} \, \|\na \Delta_k\theta\|_{L^q}.
$$
Inserting the estimates for $K_1$ and $K_2$ above in (\ref{k1k2}), we obtain
\begin{eqnarray*}
&& \frac{d}{dt} \|\Delta_j\nabla\theta\|_{L^q} + \lambda \|\Delta_j\nabla\theta\|_{L^q}
\le C\, \|\nabla u\|_{L^\infty}\, \|\Delta_j\nabla\theta\|_{L^q}\\
&& \qquad\quad + \, C\, \|\Delta_j \nabla u\|_{L^\infty}\, \sum_{m\leq j-1} \|\Delta_m \nabla\theta\|_{L^q}
+  C\, \sum_{k\ge j-1} 2^{j-k} \|\na \Delta_k u\|_{L^\infty} \, \|\na \Delta_k \theta\|_{L^q}.
\end{eqnarray*}
Summing over all integer $j$ and applying Young's inequality for series convolution, we obtain
\begin{eqnarray*}
\frac{d}{dt} \|\nabla\theta\|_{\mathring{B}^0_{q,1}} + \lambda \|\nabla\theta\|_{\mathring{B}^0_{q,1}}
&\le& C\,\|\nabla u\|_{L^\infty}\,\|\nabla\theta\|_{\mathring{B}^0_{q,1}}
+ C\,\|\nabla u\|_{\mathring{B}^0_{\infty,1}}\,\|\nabla\theta\|_{\mathring{B}^0_{q,1}}.
\end{eqnarray*}
Invoking the simple fact
$$
\|\nabla u\|_{L^\infty} \le \|\nabla u\|_{\mathring{B}^0_{\infty,1}},
$$
we obtain
\begin{eqnarray}
\frac{d}{dt} \|\nabla\theta\|_{\mathring{B}^0_{q,1}} + \lambda \|\nabla\theta\|_{\mathring{B}^0_{q,1}}
\le C\,\|\nabla u\|_{\mathring{B}^0_{\infty,1}}\,\|\nabla\theta\|_{\mathring{B}^0_{q,1}} \label{thbesov}
\end{eqnarray}
for a pure constant $C$ independent of $q$.

\vskip .1in
We now derive an differential inequality for $\|\na u\|_{\mathring{B}^0_{q,1}}$. The process is similar to
that for $\|\na \theta\|_{\mathring{B}^0_{q,1}}$, but we need to deal with the pressure term. Applying $\Delta_j$
to the equation of $\na u$ yields
$$
\partial_t\Delta_j \nabla u + \Delta_j\nabla((u\cdot \nabla)u)  +  \nu \Delta_j\nabla u
=-\Delta_j \na\na p + \Delta_j \nabla (\theta\mathbf{e}_2).
$$
Multiplying by $\Delta_j\nabla u|\Delta_j\nabla u|^{q-2}$ and integrating in space, we have
\begin{eqnarray}
\frac1q \frac{d}{dt} \|\Delta_j\nabla u\|_{L^q}^q + \nu \|\Delta_j\nabla u\|_{L^q}^q =L_1 + L_2 + L_3+ L_4, \label{L1L2L3}
\end{eqnarray}
where
\begin{eqnarray*}
L_1 &=& -\int \Delta_j\nabla u|\Delta_j\nabla u|^{q-2}\cdot\, \Delta_j(\nabla u\cdot \nabla u), \\
L_2 &=& -\int \Delta_j\nabla u|\Delta_j\nabla u|^{q-2}\cdot\, \Delta_j(u\cdot\nabla\nabla u),\\
L_3 &=& -\int \Delta_j\nabla u|\Delta_j\nabla u|^{q-2}\cdot\, \Delta_j \na\na p,\\
L_4 &=& -\int \Delta_j\nabla u|\Delta_j\nabla u|^{q-2}\cdot\,\Delta_j \nabla (\theta\mathbf{e}_2).
\end{eqnarray*}
$L_1$ and $L_2$ can be estimated in a similar fashion as $K_1$ and $K_2$, respectively.
They obey the following bounds,
\begin{eqnarray*}
|L_1|, |L_2| &\le& C\, \|\nabla u\|_{L^\infty}\, \|\Delta_j\nabla u\|_{L^q}^q + C\,\|\Delta_j\nabla u\|_{L^q}^{q-1}\,\|\Delta_j \nabla u\|_{L^\infty}\,
\|\nabla u\|_{\mathring{B}^0_{q,1}}\\
&& + \,C\,\|\Delta_j\nabla u\|_{L^q}^{q-1}\,\sum_{k\ge j-1} 2^{j-k}\,\|\Delta_k \nabla u\|_{L^\infty}\,\|\Delta_k\nabla u\|_{L^q}.
\end{eqnarray*}
To bound $L_3$, we first apply the divergence-free condition to obtain
$$
-\Delta p = \partial_l u_m \,\partial_m u_l - \partial_2 \theta,
$$
where the Einstein summation convention is invoked. Therefore,
$$
\Delta_j \na\na p = \na\na (-\Delta)^{-1} \Delta_j(\partial_l u_m \,\partial_m u_l)
-\na\na (-\Delta)^{-1} \Delta_j \partial_2 \theta.
$$
Since $\na\na (-\Delta)^{-1}$ are two Riesz transforms, $\na\na (-\Delta)^{-1} \Delta_j(\partial_l u_m \,\partial_m u_l)$
admits the same bound as $\Delta_j(\partial_l u_m \,\partial_m u_l)$ in any $L^q$ with $q\in [1,\infty]$. The reason
for the boundedness in the case of $p=1$ or $p=\infty$ is that $\Delta_j$ is a homogeneous localization operator. Therefore,
$L_3$ can be handled similarly as $L_1$ and
\begin{eqnarray*}
|L_3| &\le& C\, \|\nabla u\|_{L^\infty}\, \|\Delta_j\nabla u\|_{L^q}^q + C\,\|\Delta_j\nabla u\|_{L^q}^{q-1}\,\|\Delta_j \nabla u\|_{L^\infty}\,
\|\nabla u\|_{\mathring{B}^0_{q,1}}\\
&& + \,C\,\|\Delta_j\nabla u\|_{L^q}^{q-1}\,\sum_{k\ge j-1} 2^{j-k}\,\|\Delta_k \nabla u\|_{L^\infty}\,\|\Delta_k\nabla u\|_{L^q} + \|\Delta_j\nabla u\|_{L^q}^{q-1}\,
\|\Delta_j\nabla \theta\|_{L^q}.
\end{eqnarray*}
Applying H\"{o}lder's inequality to $L_4$ yields
$$
|L_4| \le C\, \|\Delta_j\nabla u\|_{L^q}^{q-1}\,\|\Delta_j\nabla \theta\|_{L^q}.
$$
Inserting the estimates for $|L_1|, |L_2|$, $|L_3|$ and $|L_4|$ in (\ref{L1L2L3}), we find
\begin{eqnarray*}
&& \frac{d}{dt}\|\Delta_j\nabla u\|_{L^q} + \nu \|\Delta_j\nabla u\|_{L^q}
\le C\, \|\nabla u\|_{L^\infty}\, \|\Delta_j\nabla u\|_{L^q} + C\,\|\Delta_j \nabla u\|_{L^\infty}\, \|\nabla u\|_{\mathring{B}^0_{q,1}}\\
&& \qquad\qquad + \,C\, \sum_{k\ge j-1} 2^{j-k}\,\|\Delta_k \nabla u\|_{L^\infty}\,\|\Delta_k\nabla u\|_{L^q} + C\,\|\Delta_j\nabla \theta\|_{L^q}.
\end{eqnarray*}
Summing over all integer $j$ and by Young's inequality for series convolution, we have
\begin{eqnarray}
\frac{d}{dt} \|\nabla u\|_{\mathring{B}^0_{q,1}} + \nu \|\nabla u\|_{\mathring{B}^0_{q,1}}
&\le& C_0\,\|\nabla u\|_{\mathring{B}^0_{\infty,1}}\,\|\nabla u\|_{\mathring{B}^0_{q,1}}
+ C_0\,\|\nabla \theta\|_{\mathring{B}^0_{q,1}} \label{ues}
\end{eqnarray}
for a constant $C_0$ independent of $q$, $\nu$ and $\lambda$. (\ref{thbesov}) and (\ref{ues}) yield (\ref{apin}).
This completes the proof of Proposition \ref{globalbd}.
\end{proof}

\vskip .1in
To prove Theorem \ref{main}, we need a few notation and list several facts.
For $N>0$, we denote by $J_N$ the Fourier multiplier operator defined by
$$
\widehat{J_N f}(\xi) = \chi_{B(0,N)}(\xi) \, \widehat{f}(\xi),
$$
where $B(0,N)$ denotes the closed ball centered at the origin with radius $N$ and $\chi_{B(0,N)}$
the characteristic function on $B(0,N)$. Let $\mathbb{P}$ denote the
Leray projection onto divergence-free vector fields. More precise definition of $\mathbb{P}$
can be found in the book of Majda and Bertozzi \cite[p.35, p.99]{MB}. The following simple properties
of $J_N$ and $\mathbb{P}$ will be used.

\begin{lemma} \label{jnp}
Let $J_N$ with $N>0$ and $\mathbb{P}$ denote the aforementioned operators.
Then the following properties hold:
\begin{enumerate}
\item For any $s\ge 0$,
$$
\|J_N f\|_{H^s}\leq \|f\|_{H^s}, \qquad \|\mathbb{P} f\|_{H^s}\leq \|f\|_{H^s};
$$
\item Assume that $f \in L^2$ and  $s\ge 0$, then $J_N f \in H^s$ and
$$
\|J_Nf\|_{H^{s}}\leq C\,N^{s}\|f\|_{L^2};
$$
\item Assume that $f\in H^s$, then
$$
\|J_N f-f\|_{H^{s-1}}\leq\frac{C}{N}\|f\|_{H^s}, \quad
\|J_N f-f\|_{H^s}\,\to 0 \quad\mbox{as $N\rightarrow\infty$}.
$$
\end{enumerate}
\end{lemma}

\vskip .2in
We are now ready to prove Theorem \ref{main}.
\begin{proof}[Proof of Theorem \ref{main}] The proof of this theorem is long.
For the sake of clarity, we divided it into five major steps.

\vskip .1in
{\bf Step 1. Construction of approximate solutions}. Let $N>0$ be an integer. In this step, we construct
a smooth global solution $(u^N, \theta^N)$ satisfying
\begin{equation} \label{app}
\begin{cases}
\partial_t u^N + \mathbb{P} J_N(\mathbb{P} J_N u^N\cdot\nabla \mathbb{P} J_N u^N)
  + \nu \mathbb{P} J_Nu^N = \mathbb{P} J_N(\theta^N \mathbf{e}_2), \\
\partial_t \theta^N + J_N(\mathbb{P}J_N u^N\cdot\nabla J_N \theta^N)
  + \lambda J_N \theta^N =0,\\
u^N(x,0) = J_N u_0 (x), \qquad \theta^N(x,0) = J_N \theta_0 (x).
\end{cases}
\end{equation}
It follows from Picard's Theorem (\cite[p.100]{MB}) that, for fixed $N>0$,
(\ref{app}) has a unique global smooth solution $(u^N, \theta^N)$ satisfying,
for any $T>0$,
$$
(u^N, \theta^N) \in C([0,T]; H^s(\mathbb{R}^2))
$$
for any $s>0$. In particular,
$$
\nabla u^N, \,\,\nabla \theta^N \in C([0,T]; \mathring{B}^0_{\infty,1}(\mathbb{R}^2)).
$$
It is easily checked that, if $(u^N, \theta^N)$ solves (\ref{app}),
then $(J_N u^N, J_N\theta^N)$ and $(\mathbb{P}u^N, \theta^N)$ also solve
(\ref{app}). By the uniqueness,
$$
J_N u^N =\mathbb{P}u^N =u^N, \qquad J_N \theta^N =\theta^N.
$$
$\mathbb{P}u^N =u^N$ implies $\nabla \cdot u^N=0$. Consequently (\ref{app})
is reduced to
\begin{equation} \label{appeq}
\begin{cases}
\partial_t u^N + \mathbb{P} J_N(u^N\cdot\nabla  u^N)
  + \nu u^N = \mathbb{P} J_N(\theta^N \mathbf{e}_2), \\
\partial_t \theta^N + J_N(u^N\cdot\nabla\theta^N)
  + \lambda\theta^N =0,\\
\nabla \cdot u^N=0.
\end{cases}
\end{equation}

\vskip .1in
{\bf Step 2. Uniform global bounds}.  In this step, we establish uniform global bounds for $(u^N, \theta^N)$.
Simple energy estimates combined with $\nabla\cdot u^N=0$ yields
\begin{eqnarray}
&& \|\theta^N(t)\|_{L^2} = \|J_N \theta_0\|_{L^2}\,e^{-\lambda t}
\le \|\theta_0\|_{L^2}\,e^{-\lambda t}, \nonumber\\
&&\|u^N(t)\|_{L^2} \le \|u_0\|_{L^2} \,e^{-\nu t} + \frac1{\nu} \|\theta_0\|_{L^2}. \label{L2global}
\end{eqnarray}
Furthermore, a similar procedure as in the proof of Proposition \ref{globalbd} implies
\begin{eqnarray}
&& \frac{d}{dt} \|\nabla u^N\|_{\mathring{B}^0_{\infty,1}} + \nu \|\nabla u^N\|_{\mathring{B}^0_{\infty,1}}
\le C_0\,\|\nabla u^N\|^2_{\mathring{B}^0_{\infty,1}}
+ C_0\,\|\nabla \theta^N\|_{\mathring{B}^0_{\infty,1}}, \label{diff1}\\
&& \frac{d}{dt} \|\nabla\theta^N\|_{\mathring{B}^0_{\infty,1}} + \lambda \|\nabla\theta^N\|_{\mathring{B}^0_{\infty,1}}
\le C_0\,\|\nabla u^N\|_{\mathring{B}^0_{\infty,1}}\,\|\nabla\theta^N\|_{\mathring{B}^0_{\infty,1}},
\label{diff2}
\end{eqnarray}
where $C_0$ is a constant independent of $N$, $\nu$ and $\lambda$. We claim that
these differential inequalities yield the global bounds, for large $N$ and any $t>0$,
\begin{equation} \label{locabb}
\|\nabla u^N(t)\|_{\mathring{B}^0_{\infty,1}} < A_0, \qquad \|\nabla\theta^N\|_{\mathring{B}^0_{\infty,1}} < B_0.
\end{equation}
To see this, we first choose a large $N$ such that
\begin{equation} \label{initials}
\|\nabla J_N u_0\|_{\mathring{B}^0_{\infty,1}} <A_0, \qquad \|\nabla J_N \theta_0\|_{\mathring{B}^0_{\infty,1}} < B_0.
\end{equation}
(\ref{initials}) is realized as follows. Since $u_0 \in \mathring{B}^0_{\infty,1}$ satisfies (\ref{smallness}), we can
choose $N$ sufficiently large such that
$$
\|\nabla u_0\|_{\mathring{B}^0_{\infty,1}} + \|\Delta_{j_0+1} J_N \nabla u_0\|_{L^\infty} + \|\Delta_{j_0+2} J_N \nabla u_0\|_{L^\infty} < A_0,
$$
where $j_0$ is an integer such that $2^{j_0+1} \le N <2^{j_0+2}$.
By the definition of $\mathring{B}^0_{\infty,1}$ and $J_N$,
\begin{eqnarray*}
\|\nabla J_N u_0\|_{\mathring{B}^0_{\infty,1}} &\le& \sum_{j=-\infty}^{j_0}
\|\Delta_j \nabla J_N u_0\|_{L^\infty} + \|\Delta_{j_0+1} \nabla J_N u_0\|_{L^\infty} + \|\Delta_{j_0+2} \nabla J_N u_0\|_{L^\infty} \\
&=& \sum_{j=-\infty}^{j_0}  \|\Delta_j\nabla  u_0\|_{L^\infty} + \|\Delta_{j_0+1}\nabla  J_N u_0\|_{L^\infty} + \|\Delta_{j_0+2}\nabla  J_N u_0\|_{L^\infty}\\
&<& \|\nabla u_0\|_{\mathring{B}^0_{\infty,1}} + \|\Delta_{j_0+1}\nabla  J_N u_0\|_{L^\infty} + \|\Delta_{j_0+2}\nabla  J_N u_0\|_{L^\infty}\\
&<& A_0.
\end{eqnarray*}
Here we have used the facts that $\Delta_j J_N =\Delta_j$ for $j\le j_0$ and $\Delta_{j_0+3} J_N u_0 =0$ due to
\begin{eqnarray*}
&&\widehat{\Delta_j J_N u_0}=\widehat{\Phi}_{j}(\xi)\, \chi_{B(0,N)}(\xi) \, \widehat{u}_0(\xi) =\widehat{\Phi}_{j}(\xi)\,\widehat{u}_0(\xi)
=\widehat{\Delta_j u_0},\\
&&\widehat{\Delta_{j_0+3} J_N u_0}(\xi) = \widehat{\Phi}_{j_0+3}(\xi)\, \chi_{B(0,N)}(\xi) \, \widehat{u}_0(\xi) \equiv 0.
\end{eqnarray*}
More details $\Delta_j$ and $\widehat{\Phi}_{j}$ can be found in Appendix \ref{Besov}. Similarly, for sufficiently large $N$,
$$
\|\nabla J_N \theta_0\|_{\mathring{B}^0_{\infty,1}} < B_0.
$$
Now suppose (\ref{locabb}) is not true and $T^*>0$ is the first time such that
at least one of the inequalities in (\ref{locabb}) is violated. That is,
\begin{equation}\label{ss0}
\|\nabla u^N(T^*)\|_{\mathring{B}^0_{\infty,1}} =A_0\quad \mbox{or}\quad
\|\nabla \theta^N(T^*)\|_{\mathring{B}^0_{\infty,1}} =B_0
\end{equation}
and, for $t\in (0,T^*)$,
\begin{equation}\label{ssl}
\|\nabla u^N(t)\|_{\mathring{B}^0_{\infty,1}} < A_0\quad \mbox{and}\quad
\|\nabla \theta^N(t)\|_{\mathring{B}^0_{\infty,1}} <B_0.
\end{equation}
A contradiction then easily follows from (\ref{diff1}) and (\ref{diff2}). In fact,
(\ref{diff2}), (\ref{ssl}) and the definition of $A_0$ in (\ref{smallness})  imply
\begin{eqnarray}
\|\nabla \theta^N(T^*)\|_{\mathring{B}^0_{\infty,1}} &\le& \|\nabla J_N\theta_0\|_{\mathring{B}^0_{\infty,1}}
\exp\left(-\int_0^{T^*} (\lambda-C_0\|\nabla u^N(t)\|_{\mathring{B}^0_{\infty,1}})\,dt \right) \nonumber\\
&\le&  \|\nabla J_N\theta_0\|_{\mathring{B}^0_{\infty,1}}  <B_0. \label{juu1}
\end{eqnarray}
By (\ref{diff1}), (\ref{ssl}) and the definitions of $A_0$ and $B_0$ in (\ref{smallness}),
\begin{eqnarray}
&& \|\nabla u^N(T^*)\|_{\mathring{B}^0_{\infty,1}} \le \|\nabla J_N u_0\|_{\mathring{B}^0_{\infty,1}} \exp\left(-\int_0^{T^*} (\nu-C_0\|\nabla u^N(t)\|_{\mathring{B}^0_{\infty,1}})\,dt \right)\nonumber\\
&& \qquad  + \int_0^{T^*}\exp\left(- \int_t^{T^*} (\nu-C_0\|\nabla u^N(s)\|_{\mathring{B}^0_{\infty,1}})\,ds\right) \|\nabla \theta^N(t)\|_{\mathring{B}^0_{\infty,1}}\,dt\nonumber\\
&& \qquad\le \|\nabla J_N u_0\|_{\mathring{B}^0_{\infty,1}}
\exp\left(-\nu T^*/2\right)
+ \frac{2}{\nu}\left(1-\exp\left(-\nu T^*/2\right)\right) \|\nabla J_N\theta_0\|_{\mathring{B}^0_{\infty,1}} \nonumber\\
&& \qquad < A_0 \exp\left(-\nu T^*/2\right) + A_0 \left(1-\exp\left(-\nu T^*/2\right)\right) =A_0. \label{juu2}
\end{eqnarray}
Clearly (\ref{juu1}) and (\ref{juu2}) contradict with (\ref{ss0}).

\vskip .1in
{\bf Step 3. Extraction of a strongly convergent subsequence}. We show here that we can
extract a subsequence of $(u^N, \theta^N)$, still
denoted by $(u^N, \theta^N)$, such that
\begin{equation} \label{l2con}
\|(u^N, \theta^N) -(u, \theta)\|_{L^2} \to 0, \qquad \mbox{as $N\to \infty$},
\end{equation}
where $(u, \theta) \in L^2$. This is achieved by showing that
$(u^N, \theta^N)$ is a Cauchy sequence in
$L^2$, namely
\begin{equation} \label{conv}
\|(u^N(t), \theta^N(t)) -(u^{N'}(t), \theta^{N'}(t))\|_{L^2} \to 0
\end{equation}
as $N$ and $N'$ tend to infinity. With the global bounds (\ref{L2global}) and (\ref{locabb})
at our disposal, it is not hard to verify (\ref{conv}) by performing
energy estimates on (\ref{appeq}). We omit the details.

\vskip .1in
{\bf Step 4. Verifying that $(u, \theta)$ solves (\ref{dbouss})}. Now we show
that $(u, \theta)$ solves the 2D Boussinesq equations (\ref{dbouss})
in the sense of $\mathring{H}_*^{-\sigma}$
for any $\sigma \in (0,1)$, where $\mathring{H}_*^{-\sigma}$ denotes a subspace of $\mathring{H}^{-\sigma}$ ,
$$
\mathring{H}_*^{-\sigma} = \left\{f\in \mathring{H}^{-\sigma}|\,\,\|f\|_{\mathring{H}_*^{-\sigma}} <\infty \right\}
$$
with
$$
\|f\|^2_{\mathring{H}_*^{-\sigma}} =\sum_{j=-\infty}^0 \|\Delta_j f\|^2_{L^2} + \sum_{j=1}^\infty 2^{-2\sigma j} \|\Delta_j f\|_{L^2}^2.
$$
We take the limit of (\ref{appeq}) as $N\to \infty$. Trivially
\begin{equation} \label{lincon}
\nu u^N \to \nu u, \qquad \mathbb{P} J_N(\theta^N \mathbf{e}_2) \to \mathbb{P} (\theta \mathbf{e}_2)
\qquad \mbox{in\,\, $\mathring{H}_*^{-\sigma}$}.
\end{equation}
Our main effort is devoted to showing the convergence of the nonlinear term. We consider the
difference $\|\mathbb{P} J_N(u^N\cdot\nabla  u^N) -\mathbb{P} (u\cdot\nabla u)\|_{\mathring{H}_*^{-\sigma}}$.
By Lemma \ref{jnp},
\begin{eqnarray}
&& \|\mathbb{P} J_N(u^N\cdot\nabla  u^N) -\mathbb{P} (u\cdot\nabla u)\|_{\mathring{H}_*^{-\sigma}} \nonumber\\
&& \qquad \le \|\mathbb{P} J_N(u^N\cdot\nabla  u^N -u\cdot\nabla u)\|_{\mathring{H}_*^{-\sigma}}
+ \|\mathbb{P} (J_N-1) (u\cdot\nabla u)\|_{\mathring{H}_*^{-\sigma}} \nonumber \\
&& \qquad \le \|u^N\cdot\nabla  u^N -u\cdot\nabla u\|_{\mathring{H}_*^{-\sigma}}
+ \|\mathbb{P} (J_N-1) (u\cdot\nabla u)\|_{L^2}  \nonumber \\
&& \qquad \le \|u^N-u\|_{L^2}\,\|\nabla  u^N\|_{L^\infty}
+ \|u\cdot\nabla(u^N-u)\|_{\mathring{H}_*^{-\sigma}}  \nonumber\\
&& \qquad \quad + \|(J_N-1) (u\cdot\nabla u)\|_{L^2}. \label{diff0}
\end{eqnarray}
Due to the embedding $\mathring{B}^0_{\infty,1} \hookrightarrow L^\infty$ and the bound in (\ref{locabb}),
$$
\|u\cdot\nabla u\|_{L^2} \le \|u\|_{L^2} \,\|\nabla u\|_{L^\infty} \le \|u_0\|_{L^2}\,A_0.
$$
Therefore, (\ref{locabb}), (\ref{l2con}) and Lemma \ref{jnp} imply that, as $N\to \infty$,
$$
\|u^N-u\|_{L^2}\,\|\nabla  u^N\|_{L^\infty} \le A_0 \|u^N-u\|_{L^2}\,\to 0, \qquad
\|(J_N-1) (u\cdot\nabla u)\|_{L^2}\, \to\, 0.
$$
To show that $\|u\cdot\nabla(u^N-u)\|_{\mathring{H}_*^{-\sigma}}\,\to 0$ as $N\to \infty$, we write by the notion
of paraproduct, for any integer $j$,
$$
\Delta_j(u\cdot\nabla(u^N-u)) = M_1 + M_2 + M_3,
$$
where
\begin{eqnarray*}
M_1 &=& \sum_{|j-k|\le 2} \Delta_j (S_{k-1} u \cdot \nabla \Delta_k (u^N-u)), \\
M_2 &=& \sum_{|j-k|\le 2} \Delta_j (\Delta_k u \cdot \nabla S_{k-1}(u^N-u)), \\
M_3 &=& \sum_{k\ge j-1} \Delta_j (\Delta_k u \cdot \nabla \widetilde{\Delta}_k(u^N-u)).
\end{eqnarray*}
Letting $r=\frac2\sigma$ and $\frac1q+\frac1r=\frac12$ and applying H\"{o}lder's inequality, we have
\begin{eqnarray*}
\|M_1\|_{L^2} \le C\, \|S_{j-1} u\|_{L^{q}}\, \|\nabla \Delta_j (u^N-u)\|_{L^r}.
\end{eqnarray*}
By Bernstein's inequality and an interpolation inequality,
\begin{eqnarray*}
\|M_1\|_{L^2} &\le& C\, 2^j\,\|u\|_{L^{q}}\, \|\Delta_j (u^N-u)\|_{L^r}\\
 &\le& C\, 2^j\,\|u\|_{L^{q}}\, \|\Delta_j (u^N-u)\|^{1-\frac2r}_{L^\infty} \|\Delta_j (u^N-u)\|^{\frac2r}_{L^2}.
\end{eqnarray*}
Thanks to $q>2$, (\ref{locabb}) and (\ref{l2con}), we apply Bernstein's inequality to obtain
\begin{eqnarray*}
\|u\|_{L^q} \le \|u\|_{B^0_{q,2}} &=&\|\Delta_{-1} u\|_{L^2}
+ \left[\sum_{j\ge 0} \|\Delta_j u\|_{L^q}^2\right]^{\frac12} \\
&=&\|\Delta_{-1} u\|_{L^2} + \left[\sum_{j\ge 0} \|\Delta_j u\|^{2-\frac4q}_{L^\infty}
\|\Delta_j u\|^{\frac4q}_{L^2}\right]^{\frac12} \\
&=&\|\Delta_{-1} u\|_{L^2} + \left[\sum_{j\ge 0} \|\Delta_j u\|^{2}_{L^\infty} \right]^{1-\frac2q}
\left[\sum_{j\ge 0}\|\Delta_j u\|^2_{L^2}\right]^{\frac2q}\\
&\le& C\,\|u\|_{L^2} (1+\|u\|_{\mathring{B}^0_{\infty,1}}) < \infty.
\end{eqnarray*}
By $\nabla\cdot u=0$, Bernstein's inequality and  H\"{o}lder's inequality,
\begin{eqnarray*}
\|M_2\|_{L^2} &\le& C\, 2^j \|\Delta_j u\|_{L^\infty} \,\|S_{j-1}(u^N-u)\|_{L^2} \\
&\le& C\, 2^j \|\Delta_j u\|_{L^\infty}\,\|u^N-u\|_{L^2}
\end{eqnarray*}
and
\begin{eqnarray*}
\|M_3\|_{L^2} &\le& C\,\sum_{k\ge j-1} 2^{j-k} \|\nabla \Delta_k u\|_{L^\infty}\, \|\Delta_k (u^N-u)\|_{L^2}.
\end{eqnarray*}
Therefore, by combining the bounds above, we have
\begin{eqnarray}
&& \|u\cdot\nabla(u^N-u)\|^2_{{\mathring{H}_*^{-\sigma}}} \nonumber\\
&& \qquad = \sum_{j=-\infty}^0 \|\Delta_j(u\cdot\nabla(u^N-u))\|^2_{L^2}
+ \sum_{j=1}^\infty 2^{-2\sigma j}\|\Delta_j(u\cdot\nabla(u^N-u))\|^2_{L^2} \nonumber\\
&& \qquad \le C\,\|u^N-u\|_{L^2}^2\, \sum_{j=-\infty}^\infty 2^{2j} \|\Delta_j u\|^2_{L^\infty} \nonumber\\
&&\qquad \quad + \,C\, \sum_{j=-\infty}^\infty\left[\sum_{k\ge j-1} 2^{j-k} \|\nabla \Delta_k u\|_{L^\infty}\,
\|\Delta_k (u^N-u)\|_{L^2}\right]^2\nonumber\\
&&\qquad \quad + \,C\, \|u\|^2_{L^{q}}\,\sum_{j=-\infty}^\infty 2^{2(1-\sigma)j}
\|\Delta_j (u^N-u)\|^{2-\frac4r}_{L^\infty} \|\Delta_j (u^N-u)\|^{\frac4r}_{L^2}. \label{dd0}
\end{eqnarray}
We further estimate the terms on the right and have
\begin{eqnarray}
&&\sum_{j=-\infty}^\infty 2^{2j} \|\Delta_j u\|^2_{L^\infty}
= \|u\|^2_{\mathring{B}^1_{\infty,2}}\le C\,\|\nabla u\|^2_{\mathring{B}^0_{\infty,1}} <\infty, \label{ddee0}\\
&&\sum_{j=-\infty}^\infty\left[\sum_{k\ge j-1} 2^{j-k} \|\nabla \Delta_k u\|_{L^\infty}\,
\|\Delta_k (u^N-u)\|_{L^2}\right]^2 \nonumber\\
&&\qquad \le \sum_{j=-\infty}^\infty\|\nabla \Delta_j u\|^2_{L^\infty}\,
\|\Delta_j (u^N-u)\|^2_{L^2} \le C\,\|u^N-u\|_{L^2}^2\, \|\nabla u\|^2_{\mathring{B}^0_{\infty,1}}, \label{ddee1}
\end{eqnarray}
where the Besov embedding $\mathring{B}^0_{\infty,1}\hookrightarrow \mathring{B}^0_{\infty,2}$ is used
in the first inequality and Young's inequality for series convolutions is used in the second inequality.
Noticing that $\frac2r=\sigma$, we obtain
by H\"{o}lder's inequality
\begin{eqnarray}
&& \sum_{j=-\infty}^\infty 2^{2(1-\sigma)j}
 \|\Delta_j (u^N-u)\|^{2-\frac4r}_{L^\infty} \|\Delta_j (u^N-u)\|^{\frac4r}_{L^2} \nonumber \\
&& \qquad \le \left[\sum_{j=-\infty}^\infty 2^{2j} \|\Delta_j (u^N-u)\|^{2}_{L^\infty}\right]^{1-\sigma}
\left[\sum_{j=-\infty}^\infty \|\Delta_j (u^N-u)\|^2\right]^\sigma \nonumber\\
&& \qquad \le \|u^N-u\|^{1-\sigma}_{\mathring{B}^1_{\infty,2}}\, \|u^N-u\|^\sigma_{L^2} \nonumber\\
&& \qquad \le C\,\left(\|\nabla u^N\|_{\mathring{B}^0_{\infty,1}} + \|\nabla u\|_{\mathring{B}^0_{\infty,1}}\right)^{1-\sigma}
\, \|u^N-u\|^\sigma_{L^2}. \label{ddee2}
\end{eqnarray}
Inserting (\ref{ddee0}), (\ref{ddee1}) and (\ref{ddee2}) in (\ref{dd0}), we obtain
\begin{eqnarray*}
&& \|u\cdot\nabla(u^N-u)\|^2_{{\mathring{H}_*^{-\sigma}}} \le C\, \|u^N-u\|_{L^2}^2
\|\nabla u\|^2_{\mathring{B}^0_{\infty,1}}\\
&& \qquad + \,C\,\|u\|^2_{L^{q}}\,\left(\|\nabla u^N\|_{\mathring{B}^0_{\infty,1}}
+ \|\nabla u\|_{\mathring{B}^0_{\infty,1}}\right)^{1-\sigma}
\, \|u^N-u\|^\sigma_{L^2}.
\end{eqnarray*}
Therefore, (\ref{l2con}) implies
$$
\|u\cdot\nabla(u^N-u)\|^2_{{\mathring{H}_*^{-\sigma}}} \to 0, \qquad\mbox{as $N\to \infty$}.
$$
It then follows from (\ref{diff0}) that
\begin{equation} \label{noncon}
\|\mathbb{P} J_N(u^N\cdot\nabla  u^N) -\mathbb{P} (u\cdot\nabla u)\|_{\mathring{H}_*^{-\sigma}}
\to 0, \qquad\mbox{as $N\to \infty$}.
\end{equation}
As a consequence of (\ref{lincon}) and (\ref{noncon}), as $N\to \infty$,
$$
\partial_t u^N =- \mathbb{P} J_N(u^N\cdot\nabla  u^N)
 -\nu u^N + \mathbb{P} J_N(\theta^N \mathbf{e}_2)
$$
converges strongly in $\mathring{H}_*^{-\sigma}$. On the other hand, since $u^N \to u$ in $L^2$,
$$
\partial_t u^N \to \partial_t u
$$
in the distributional sense. Therefore,
\begin{equation} \label{ptcon}
\|\partial_t u^N - \partial_t u\|_{\mathring{H}_*^{-\sigma}}
\to 0, \qquad\mbox{as $N\to \infty$}.
\end{equation}
In summary, we have shown that, by letting $N\to \infty$ in (\ref{appeq})
and invoking the limits in (\ref{lincon}), (\ref{noncon}) and (\ref{ptcon}),
$$
\partial_t u + \mathbb{P} (u\cdot\nabla u) + \nu u = \mathbb{P} (\theta \mathbf{e}_2)
\qquad \mbox{in\,\, $\mathring{H}_*^{-\sigma}$},
$$
which can also be written as
$$
\partial_t u + u\cdot\nabla u + \nu u = -\nabla p + \theta \mathbf{e}_2, \quad \nabla\cdot u=0.
$$
In a similar manner, we can also show that
$$
\partial_t \theta + u\cdot\nabla \theta + \lambda\, \theta = 0 \qquad \mbox{in\,\, $\mathring{H}_*^{-\sigma}$}.
$$

\vskip .1in
{\bf Step 5. Uniqueness}. This step is devoted to showing that any two solutions $(u^{(1)}, \theta^{(1)})$
and $(u^{(2)}, \theta^{(2)})$ obeying (\ref{regclass}) must coincide.
It is clear that the difference $(v, \Theta)$ with
$$
v= u^{(1)} - u^{(2)}, \quad \Theta =\theta^{(1)}-\theta^{(2)}
$$
satisfies
\begin{equation} \label{diffeq}
\begin{cases}
\partial_t  v + \mathbb{P}(u^{(1)}\cdot\nabla v) + \mathbb{P}(v \cdot\nabla u^{(2)}) + \nu\,v =\mathbb{P}(\Theta \mathbf{e}_2),\\
\partial_t  \Theta + u^{(1)}\cdot\nabla \Theta + v \cdot\nabla \theta^{(2)} + \lambda\,\Theta =0,\\
v(x,0) =0, \quad \Theta(x,0) =0.
\end{cases}
\end{equation}
Taking the inner product with $(v, \Theta)$ yields
\begin{eqnarray*}
&& \frac12 \frac{d}{dt} \left(\|v\|_{L^2}^2 + \|\Theta\|_{L^2}^2\right) + \nu\, \|v\|_{L^2}^2
+ \lambda\,\|\Theta\|_{L^2}^2 \\
&& \qquad \le \|v\|_{L^2}\,\|\Theta\|_{L^2} + \left|\int v \cdot\nabla u^{(2)}\cdot v\right|
+ \left|\int v \cdot\nabla \theta^{(2)} \,\theta\right|.
\end{eqnarray*}
Bounding the last two terms on the right-hand side by H\"{o}ler's inequality and applying the embedding
$\mathring{B}^0_{\infty,1} \hookrightarrow L^\infty$ yield
\begin{eqnarray*}
&& \frac12 \frac{d}{dt} \left(\|v\|_{L^2}^2 + \|\Theta\|_{L^2}^2\right) + \nu\, \|v\|_{L^2}^2
+ \lambda\,\|\Theta\|_{L^2}^2 \\
&& \qquad \le C\,\left(1+ \|(\nabla u^{(2)}, \nabla\theta^{(2)})\|_{\mathring{B}^0_{\infty,1}}\right)
\left(\|v\|_{L^2}^2 + \|\Theta\|_{L^2}^2\right).
\end{eqnarray*}
Gronwall's inequality then implies that $(v, \Theta) \equiv 0$. We have completed all the steps and thus the whole
proof of Theorem \ref{main}.
\end{proof}

\vskip .3in
\appendix
\section{Functional spaces}
\label{Besov}

This appendix provides the definitions of some of the functional spaces and related facts
used in the previous sections. Materials presented in this appendix can be found in several
books and many papers (see, e.g., \cite{BCD,BL,MWZ,RS,Tri}).

\vskip .1in
We start with several notation. $\mathcal{S}$ denotes
the usual Schwarz class and ${\mathcal S}'$ its dual, the space of
tempered distributions. ${\mathcal S}_0$ denotes a subspace of ${\mathcal
S}$ defined by
$$
{\mathcal S}_0 = \left\{ \phi\in {\mathcal S}: \,\, \int_{\mathbb{R}^d}
\phi(x)\, x^\gamma \,dx =0, \,|\gamma| =0,1,2,\cdots \right\}
$$
and ${\mathcal S}_0'$ denotes its dual. ${\mathcal S}_0'$ can be identified
as
$$
{\mathcal S}_0' = {\mathcal S}' / {\mathcal S}_0^\perp = {\mathcal S}' /{\mathcal P}
$$
where ${\mathcal P}$ denotes the space of multinomials.

\vspace{.1in} To introduce the Littlewood-Paley decomposition, we
write for each $j\in \mathbb{Z}$
\begin{equation*}\label{aj}
A_j =\left\{ \xi \in \mathbb{R}^d: \,\, 2^{j-1} \le |\xi| <
2^{j+1}\right\}.
\end{equation*}
The Littlewood-Paley decomposition asserts the existence of a
sequence of functions $\{\Phi_j\}_{j\in {\mathbb Z}}\in {\mathcal S}$ such
that
$$
\mbox{supp} \widehat{\Phi}_j \subset A_j, \qquad
\widehat{\Phi}_j(\xi) = \widehat{\Phi}_0(2^{-j} \xi)
\quad\mbox{or}\quad \Phi_j (x) =2^{jd} \Phi_0(2^j x),
$$
and
$$
\sum_{j=-\infty}^\infty \widehat{\Phi}_j(\xi) = \left\{
\begin{array}{ll}
1&,\quad \mbox{if}\,\,\xi\in {\mathbb R}^d\setminus \{0\},\\
0&,\quad \mbox{if}\,\,\xi=0.
\end{array}
\right.
$$
Therefore, for a general function $\psi\in {\mathcal S}$, we have
$$
\sum_{j=-\infty}^\infty \widehat{\Phi}_j(\xi)
\widehat{\psi}(\xi)=\widehat{\psi}(\xi) \quad\mbox{for $\xi\in {\mathbb
R}^d\setminus \{0\}$}.
$$
In addition, if $\psi\in {\mathcal S}_0$, then
$$
\sum_{j=-\infty}^\infty \widehat{\Phi}_j(\xi)
\widehat{\psi}(\xi)=\widehat{\psi}(\xi) \quad\mbox{for any $\xi\in
{\mathbb R}^d $}.
$$
That is, for $\psi\in {\mathcal S}_0$,
$$
\sum_{j=-\infty}^\infty \Phi_j \ast \psi = \psi
$$
and hence
$$
\sum_{j=-\infty}^\infty \Phi_j \ast f = f, \qquad f\in {\mathcal S}_0'
$$
in the sense of weak-$\ast$ topology of ${\mathcal S}_0'$. For
notational convenience, we define
\begin{equation}\label{del1}
\mathring{\Delta}_j f = \Phi_j \ast f, \qquad j \in {\mathbb Z}.
\end{equation}

\begin{define}
For $s\in {\mathbb R}$ and $1\le p,q\le \infty$, the homogeneous Besov
space $\mathring{B}^s_{p,q}$ consists of $f\in {\mathcal S}_0' $
satisfying
$$
\|f\|_{\mathring{B}^s_{p,q}} \equiv \|2^{js} \|\mathring{\Delta}_j
f\|_{L^p}\|_{l^q} <\infty.
$$
\end{define}

\vspace{.1in}
We now choose $\Psi\in {\mathcal S}$ such that
$$
\widehat{\Psi} (\xi) = 1 - \sum_{j=0}^\infty \widehat{\Phi}_j (\xi),
\quad \xi \in {\mathbb R}^d.
$$
Then, for any $\psi\in {\mathcal S}$,
$$
\Psi \ast \psi + \sum_{j=0}^\infty \Phi_j \ast \psi =\psi
$$
and hence
\begin{equation*}\label{sf}
\Psi \ast f + \sum_{j=0}^\infty \Phi_j \ast f =f
\end{equation*}
in ${\mathcal S}'$ for any $f\in {\mathcal S}'$. To define the inhomogeneous Besov space, we set
\begin{equation} \label{del2}
\Delta_j f = \left\{
\begin{array}{ll}
0,&\quad \mbox{if}\,\,j\le -2, \\
\Psi\ast f,&\quad \mbox{if}\,\,j=-1, \\
\Phi_j \ast f, &\quad \mbox{if} \,\,j=0,1,2,\cdots.
\end{array}
\right.
\end{equation}
\begin{define}
The inhomogeneous Besov space $B^s_{p,q}$ with $1\le p,q \le \infty$
and $s\in {\mathbb R}$ consists of functions $f\in {\mathcal S}'$
satisfying
$$
\|f\|_{B^s_{p,q}} \equiv \|2^{js} \|\Delta_j f\|_{L^p} \|_{l^q}
<\infty.
$$
\end{define}

\vskip .1in
The Besov spaces $\mathring{B}^s_{p,q}$ and $B^s_{p,q}$ with  $s\in (0,1)$ and $1\le p,q\le \infty$ can be equivalently defined by the norms
$$
\|f\|_{\mathring{B}^s_{p,q}}  = \left(\int_{\mathbb{R}^d} \frac{(\|f(x+t)-f(x)\|_{L^p})^q}{|t|^{d+sq}} dt\right)^{1/q},
$$
$$
\|f\|_{B^s_{p,q}}  = \|f\|_{L^p} + \left(\int_{\mathbb{R}^d} \frac{(\|f(x+t)-f(x)\|_{L^p})^q}{|t|^{d+sq}} dt\right)^{1/q}.
$$
When $q=\infty$, the expressions are interpreted in the normal way.

\vskip .1in
Many frequently used function spaces are special cases of Besov spaces. The following proposition
lists some useful equivalence and embedding relations.
\begin{prop}
For any $s\in \mathbb{R}$,
$$
\mathring{H}^s \sim \mathring{B}^s_{2,2}, \quad H^s \sim B^s_{2,2}.
$$
For any $s\in \mathbb{R}$ and $1<q<\infty$,
$$
\mathring{B}^{s}_{q,\min\{q,2\}} \hookrightarrow \mathring{W}_{q}^s \hookrightarrow \mathring{B}^{s}_{q,\max\{q,2\}}.
$$
In particular, $\mathring{B}^{0}_{q,\min\{q,2\}} \hookrightarrow L^q \hookrightarrow \mathring{B}^{0}_{q,\max\{q,2\}}$.
\end{prop}

\vskip .1in
For notational convenience, we write $\Delta_j$ for
$\mathring{\Delta}_j$. There will be no confusion if we keep in mind that
$\Delta_j$'s associated with the homogeneous Besov spaces is defined in
(\ref{del1}) while those associated with the inhomogeneous Besov
spaces are defined in (\ref{del2}). Besides the Fourier localization operators $\Delta_j$,
the partial sum $S_j$ is also a useful notation. For an integer $j$,
$$
S_j \equiv \sum_{k=-1}^{j-1} \Delta_k,
$$
where $\Delta_k$ is given by (\ref{del2}). For any $f\in \mathcal{S}'$, the Fourier
transform of $S_j f$ is supported on the ball of radius $2^j$.

\vskip .1in
Bernstein's inequalities are useful tools in dealing with Fourier localized functions
and these inequalities trade integrability for derivatives. The following proposition
provides Bernstein type inequalities for fractional derivatives.
\begin{prop}\label{bern}
Let $\alpha\ge0$. Let $1\le p\le q\le \infty$.
\begin{enumerate}
\item[1)] If $f$ satisfies
$$
\mbox{supp}\, \widehat{f} \subset \{\xi\in \mathbb{R}^d: \,\, |\xi|
\le K 2^j \},
$$
for some integer $j$ and a constant $K>0$, then
$$
\|(-\Delta)^\alpha f\|_{L^q(\mathbb{R}^d)} \le C_1\, 2^{2\alpha j +
j d(\frac{1}{p}-\frac{1}{q})} \|f\|_{L^p(\mathbb{R}^d)}.
$$
\item[2)] If $f$ satisfies
\begin{equation*}\label{spp}
\mbox{supp}\, \widehat{f} \subset \{\xi\in \mathbb{R}^d: \,\, K_12^j
\le |\xi| \le K_2 2^j \}
\end{equation*}
for some integer $j$ and constants $0<K_1\le K_2$, then
$$
C_1\, 2^{2\alpha j} \|f\|_{L^q(\mathbb{R}^d)} \le \|(-\Delta)^\alpha
f\|_{L^q(\mathbb{R}^d)} \le C_2\, 2^{2\alpha j +
j d(\frac{1}{p}-\frac{1}{q})} \|f\|_{L^p(\mathbb{R}^d)},
$$
where $C_1$ and $C_2$ are constants depending on $\alpha,p$ and $q$
only.
\end{enumerate}
\end{prop}

\vskip .4in
\section*{Acknowledgements}
Cao was partially supported by NSF grant DMS1109022. Wu was partially supported
by NSF grant DMS1209153. Xu was partially supported by NSFC (No.11171026),
BNSF (No.2112023) and the Fundamental Research Funds for the Central Universities of China.

\end{document}